\documentclass[oneside]{amsart}
\usepackage[utf8]{inputenc}
\usepackage{graphicx}
\usepackage{amsmath,amssymb}
\usepackage{setspace}
\usepackage{float}
\usepackage{geometry}
\usepackage{caption}
\usepackage{amssymb}
\usepackage{listings}
\usepackage{amsthm}
\usepackage{graphicx}
\usepackage{biblatex}
\usepackage{amsthm}
\usepackage{hyperref}
\usepackage{lineno}
\newtheorem{theorem}{Theorem}
\newtheorem{proposition}[theorem]{Proposition}

\newtheorem{conjecture}[theorem]{Conjecture}
\newtheorem{lemma}[theorem]{Lemma}
\theoremstyle{definition}
\newtheorem{definition}[theorem]{Definition}

\title[The PCSC is true for the Kinoshita-Terasaka and Conway knot families]{The purely cosmetic surgery conjecture is true for the Kinoshita-Terasaka and Conway knot families}
\author{Bryan Boehnke, Conan Gillis, Hanwen Liu, Shuhang Xue}

\begin{document}
\begin{abstract}
    We show that all nontrivial members of the Kinoshita-Terasaka and Conway knot families satisfy the purely cosmetic surgery conjecture.
\end{abstract}
\maketitle

\section{The purely cosmetic surgery conjecture}
Let $K \subset S^3$ be a knot, and let $S^3_r(K)$ be the 3-manifold resulting from Dehn surgery on $K$ with surgery coefficient $r$. We consider the following conjecture:

\begin{conjecture}[Purely cosmetic surgery conjecture]
If $K$ is a nontrivial knot, then $S^3_r(K)$ and $S^3_s(K)$ are orientation-preserving diffeomorphic if and only if $r=s$.
\end{conjecture}

We will refer to this in the sequel as the PCSC. It has been verified for several families of knots, recently including pretzel knots (\cite{StipsiczSzabo}) and cable knots (\cite{Tao}). Our main result is the verification of the PCSC for the families of Kinoshita-Terasaka knots and their mutants, Conway knots:

\begin{theorem}\label{thm:PCSCKTConway}
    The purely cosmetic surgery conjecture holds for all nontrivial Kinoshita-Terasaka and Conway knots.
\end{theorem}
\noindent \textbf{Acknowledgments:} This paper was completed through the BSM Summer Undergraduate Research Program under the supervision of Dr. András Stipsicz.

\section{Knot invariants and the PCSC}
\subsection{The Alexander polynomial}
For a knot $K \subset S^3$, its Alexander polynomial is an integral Laurent polynomial $\Delta_K (t) \in \mathbb{Z}[t, t^{-1}]$, constructed using a diagram for $K.$ To begin this construction, we first define a marked knot diagram $(D,p)$ to be an oriented knot diagram $D$ with a distinguished edge $e$ marked by a point $p$ on $e$ (Figure 1 shows a marked knot diagram for the left-handed trefoil). 

\begin{figure}
    \centering
    \includegraphics[scale=0.20]{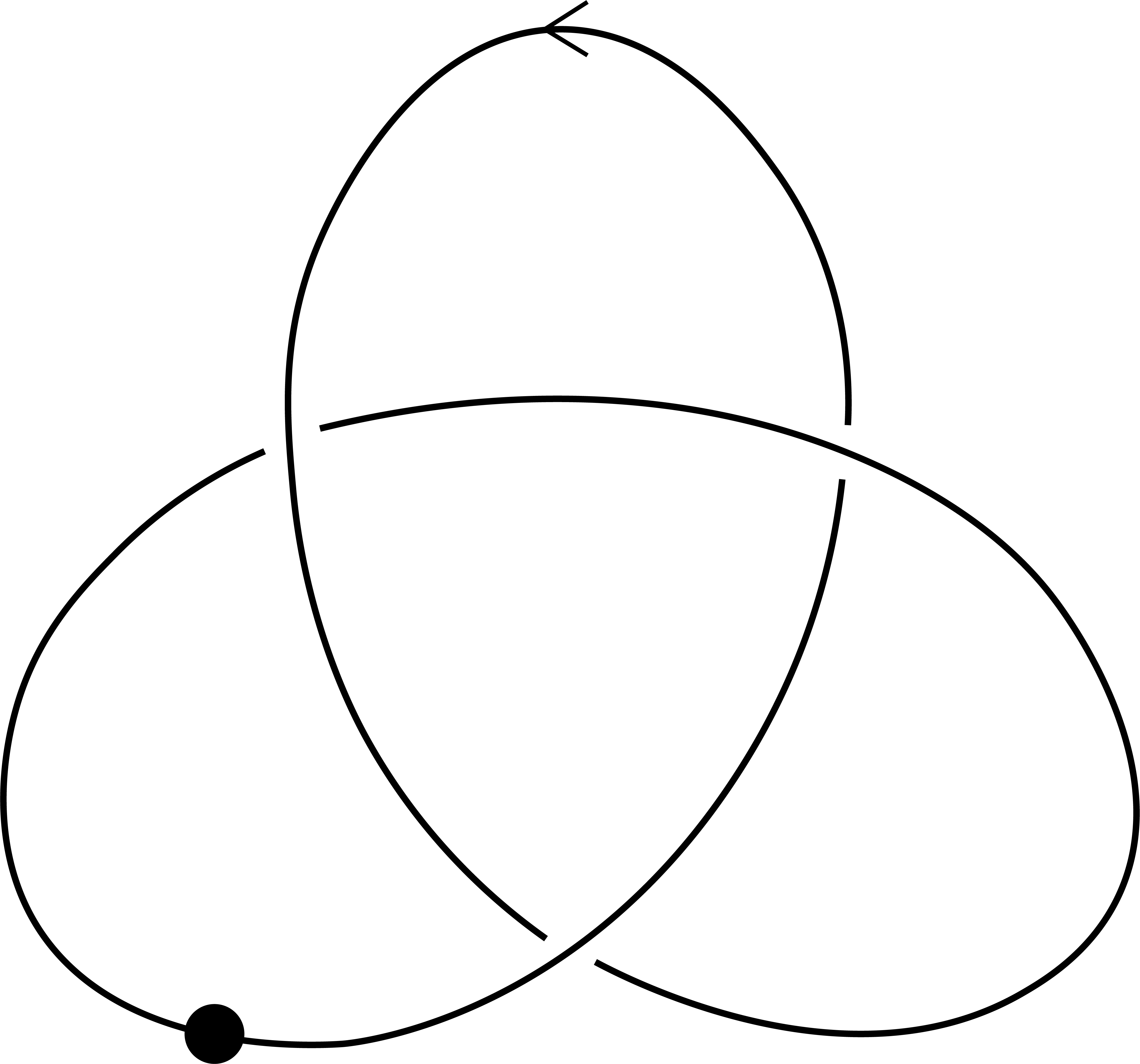}
    \caption{A marked knot diagram of the left-handed trefoil knot.}
    \label{fig:trefoil}
\end{figure}

In such a marked diagram, let $Cr(D,p)$ denote the set of crossings and let $Dom(D,p)$ denote the set of domains in the plane which do not have the marking $p$ on their boundary. Note that for a diagram with $n$ crossings, there are $n+2$ domains with $n$ in $Dom(D,p)$ disjoint from the marking $p$.

\begin{definition}
A Kauffman state $\kappa$ is defined as a bijection $\kappa: \emph{Cr}(D,p) \rightarrow \emph{Dom}(D,p)$ in a marked knot diagram $(D, p)$ such that $\kappa(c)$ is a quadrant around $c \in Cr(D,p).$ We denote $K(D,p)$ as the set of all such Kauffman states for a marked diagram $(D,p).$
\end{definition}

To a Kauffman state $\kappa$, we associate two integer quantities: 
\begin{equation} \label{localGradings}
    A(\kappa) = \sum_{c_i \in Cr(D)} A(\kappa(c_i))\qquad
    M(\kappa) = \sum_{c_i \in Cr(D)} M(\kappa(c_i)),
\end{equation}
where the local contributions $A(\kappa(c_i)) \in \{0, \pm{\frac{1}{2}}\}$ and $M(\kappa(c_i)) \in \{0, \pm{1}\}$ are defined in Figure \ref{fig:A and M gradings} for each crossing $c_i \in Cr(D,p)$. $A(\kappa)$ and $M(\kappa)$ are called the Alexander and Maslov gradings respectively.

\begin{figure}
    \centering
    \includegraphics[scale=0.35]{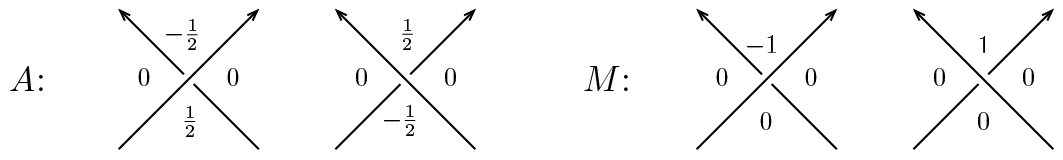}
    \caption{Local contributions to the Alexander and Maslov gradings at a crossing.}
    \label{fig:A and M gradings}
\end{figure}

\begin{definition} \label{def:alexander}
The Alexander polynomial is defined as
\begin{equation}\label{AlexanderPolynomial}
    \Delta_K (t) = \sum_{\kappa \in K(D,p)} (-1)^{M(\kappa)} \cdot t^{A(\kappa)}.
\end{equation}
\end{definition}

It can be shown that this construction yields the same polynomial for all marked diagrams $(D,p)$ of the same knot and all choices of marking $p$ on such diagrams. Hence, the Alexander polynomial is indeed a knot invariant.

\begin{figure}
    \centering
    \includegraphics[width=7cm]{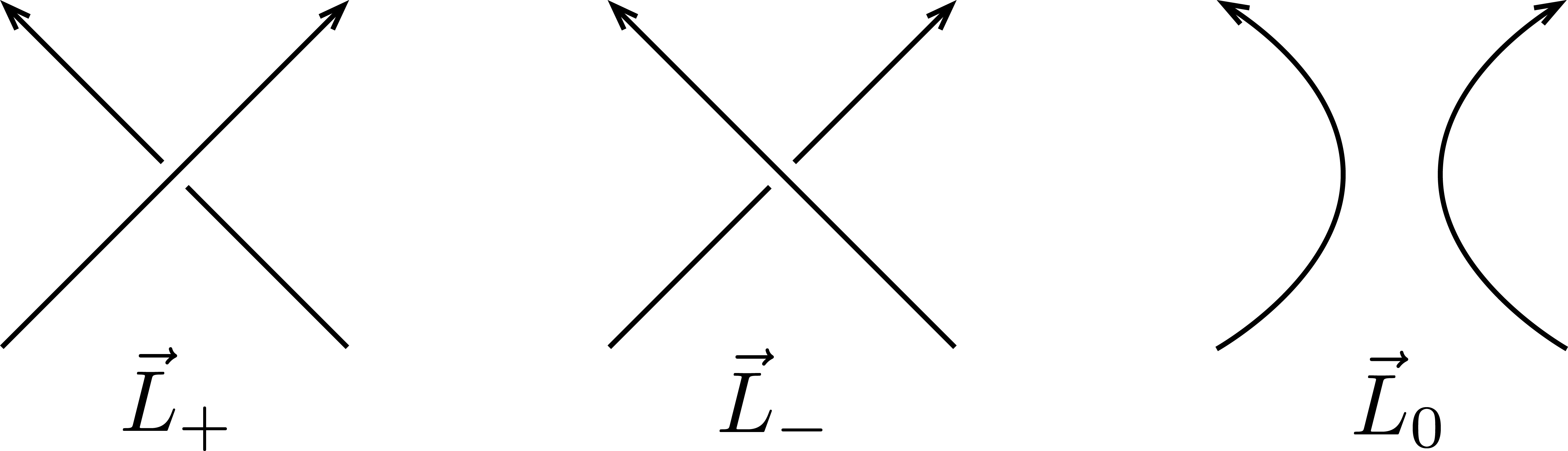}
    \caption{A skein triple.}
    \label{fig:skein0relation}
\end{figure}

Furthermore, for three oriented links $\Vec{L}_{+}$, $\Vec{L}_{-},$ and $\Vec{L}_{0}$  which correspond to diagrams that differ only at a single crossing (see Figure \ref{fig:skein0relation}), the Alexander polynomial satisfies the following equation called the \textit{skein relation}.
\begin{equation}\label{SkeinRelationDelta}
\Delta_{\Vec{L}_{+}(t)} - \Delta_{\Vec{L}_{-}}(t) = (t^{1/2} - t^{-1/2})\cdot \Delta_{\Vec{L}_{0}}(t).
\end{equation}
Together with the normalization $\Delta_{\emph{u}} = 1$ for the unknot $u$, the skein relation can be used to give an equivalent definition of the Alexander polynomial, which is often easier to compute. 

In \cite{KinoshitaTerasaka}, the Kinoshita-Terasaka knots were introduced as a family of nontrivial knots that share Alexander polynomial with the unknot, that is, $\Delta_K(t) = 1$. The Conway knots, related to the Kinoshita-Terasaka knots by mutation, share this property.

\subsection{The Jones polynomial}
Similar to the Alexander polynomial, the Jones polynomial is also a Laurent polynomial knot invariant in $\mathbb{Z}[q,q^{-1}].$  Like the Alexander polynomial, the Jones polynomial can be defined with the use of a skein relation, combined with a normalization.

\begin{definition}
For a knot $K$, the Jones polynomial is the unique Laurent polynomial $V_K(q) \in\mathbb{Z}[q,q^{-1}]$ satisfying the skein relation
\begin{equation}\label{SkeinRelationV}
q^2V_{\Vec{L}_{+}}(q) - q^2V_{\Vec{L}_{-}}(q) = (q^{-1} - q)\cdot V_{\Vec{L}_{0}}(q)
\end{equation}
for links $\Vec{L}_+,\Vec{L}_-,\Vec{L}_0$ defined above and the normalization $V_{\emph{u}} = 1$, where $u$ is the unknot.
\end{definition}

It can be shown that this polynomial is independent of the choice of the diagram for $K.$ Additionally, the Jones polynomial is more sensitive at detecting mirrors of knots, which the Alexander polynomial cannot. Specifically, if we denote the mirror of a knot $K$ as $m(K)$, then the Jones polynomial satisfies $V_{m(K)}(q) = V_K(q^{-1})$. This relationship is used later in our paper.

Another useful relationship connects values of the Alexander and Jones polynomials in the following way.

\begin{lemma}[Ichihara and Wu, \cite{IchiharaWu}]\label{relationship}
For all knots $K \subset S^3$,
\begin{equation}\label{V&Delta Relationship}
V_K^{\prime \prime}(1) = -3 \Delta_K^{\prime \prime}(1).
\end{equation}
\end{lemma}

As a final note, Ichihara and Wu developed the following sufficient condition on the Jones polynomial for the PCSC to hold for $K$:

\begin{theorem}[Ichihara and Wu, \cite{IchiharaWu}]\label{IchiharaWuThm}
For a knot $K$, if either $V_K''(1)\neq 0$ or $V_K'''(1)\neq 0,$ then the PCSC holds true.
\end{theorem}

\subsection{Knot Floer homology and thickness} \label{section:homology}
While the Alexander polynomial is a useful knot invariant, we are aided by a more sensitive invariant when verifying the PCSC for the Kinoshita-Terasaka and Conway knots. This invariant is knot Floer homology, first developed by Ozsváth and Szabó in \cite{OzsvathAndSzabo}. 

In Definition \ref{def:alexander}, we can note that information from the Maslov and Alexander gradings are contained within the definition of the Alexander polynomial. Given a marked diagram $(D,p)$ for a knot $K,$ we can similarly construct the two-variable polynomial
\begin{equation}
    G(s,t) = \sum\limits_{\kappa\in K(D,p)} s^{M(\kappa)} \cdot t^{A(\kappa)}.
\end{equation}
We note that this more sensitive polynomial contains all of the information of the Alexander polynomial, since $G(-1,t) = \Delta_K(t)$. However, $G(s,t)$ is not a knot invariant. For instance, $G(1,1)$ represents the number of Kauffman states, a quantity which can vary based on the diagram of the knot.

To obtain a knot invariant from $G(s,t)$, we consider the vector space $C_{D,p}$ generated over the field $\mathbb{F}$ of two elements by the Kauffman states of the diagram $(D,p)$.  When the basis is equipped with the bigrading $(M,A)$, it becomes a bigraded vector space with graded dimension $G(s,t)$. In \cite{OzsvathAndSzabo}, Ozsváth and Szabó constructed a boundary map to establish a chain complex $(C_{D,p}, \partial)$ and define the homology group.

\begin{theorem}[Ozsváth and Szabó, \cite{OzsvathAndSzabo}]\label{OzsvathAndSzabo}
There exists a linear map  $\partial: C_{D,p} \longrightarrow C_{D,p}$ such that $\partial^2=0$ and the bigraded homology group $H(C_{D,p}, \partial) = Ker \partial / Im \partial$ is a knot invariant.
\end{theorem}

The homology group $\widehat{\text{HFK}}(K) = H(C_{D,p}, \partial)$ is called the knot Floer homology group of the knot $K$. 

\begin{definition}
The $\delta$-grading of a homogeneous element $x \in \widehat{\text{HFK}}(K)$ of bidegree $(M,A)$ is defined as the difference between the Maslov and Alexander gradings
\begin{equation}
    \delta(x) = M(x) - A(x).
\end{equation}
\end{definition}

We observe that $\delta(x)$ is an integer, as it is defined as the difference of the integer-valued Alexander and Maslov gradings. This $\delta$-grading allows us to consider the graded vector space $\widehat{\text{HFK}}^{\delta}(K)$, and from this $\delta$-graded knot Floer homology, we can introduce another new knot invariant:

\begin{definition} \label{def:thickness}
The \textit{thickness of a knot} \emph{K} is given by
\begin{center}
    $th(K) = max\{|\delta(x) - \delta(x')| \mid x, x' \in \widehat{\text{HFK}}^\delta (K)\; homogeneous\}$.
\end{center}
\end{definition} 

Since the knot Floer homology group is itself a knot invariant, we have that the $\delta$-graded knot Floer homology group and the thickness of a knot are knot invariants as well. Also, we can note that the thickness of the knot is an integer from its definition as the difference of integer-valued $\delta$-gradings.
 
We can similarly associate a quantity of thickness to the chain complex $(C_{D,p} , \partial)$, defined as the maximal difference between $\delta$-gradings of elements of the chain complex.

\begin{definition}
The \textit{thickness of a chain complex} $(C_{D,p}, \partial)$ (generated by Kauffman states of a marked diagram $(D,p)$ for a knot $K$ over the field $\mathbb{F}$) is given by
\begin{center}
    $th(C_{D,p}) = max\{|\delta(x) - \delta(x')| \mid x, x' \in C_{D,p}\; homogeneous\}$.
\end{center}
\end{definition}

The thickness of the chain complex is not a knot invariant as it is dependent on the diagram of the knot, although it can be used to bound the thickness of the knot in the following way:

\begin{lemma} \label{lemma:thicknessineq}
Let $th(C_{D,p})$ be the thickness of a chain complex $(C_{D,p}, \partial)$ (generated by Kauffman states of a marked diagram $(D,p)$ for a knot $K$ over the field $\mathbb{F}$) and let $th(K)$ be the thickness of the same knot $K$. Then $th(C_{D,p}) \geq th(K)$.
\end{lemma}
\begin{proof}
We note that the thickness of a knot is based on the $\delta$-graded knot Floer homology which involves a subcomplex of the chain complex $(C_{D,p}, \partial)$ (of cycles). As this subcomplex has potentially fewer elements $x$ and $x'$ to consider in the definition of the thickness of a chain complex, this subcomplex can potentially have smaller thickness. Furthermore, the $\delta$-graded knot Floer homology is the result of taking the quotient of the subcomplex of cycles with the subcomplex of boundaries, which can similarly decrease the thickness. Therefore, $th(C_{D,p}) \geq th(K)$. 
\end{proof}

\begin{figure}
    \centering
    \includegraphics[scale=0.3]{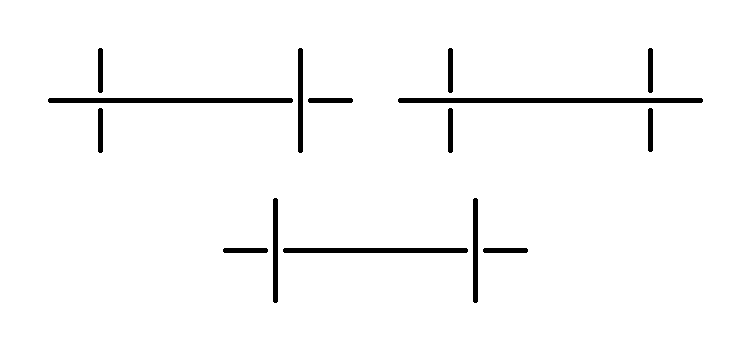}
    \caption{A horizontal good edge in the top left and two horizontal bad edges.}
    \label{fig:goodbadedges}
\end{figure}

Let $(D, p)$ be a marked knot diagram for a knot $K$. Notice that for a domain $\mathcal{D}$ in the diagram, the boundary of $\mathcal{D}$ is comprised of a sequence of arcs in the knot, called edges, with exactly two crossings at the two ends of each edge. We define an edge to be good if the two crossings are different when traversed along the edge, that is, the edge crosses over at one crossing and under at the other. An edge is defined to be bad otherwise. Figure \ref{fig:goodbadedges} gives examples of good and bad edges. A domain is good if all of the edges on its boundary are good, and a domain is bad if there is a bad edge on its boundary.

Using these bad domains, Stipsicz and Szabó were able to bound the thickness of a chain complex above by half the number of bad domains. With Lemma \ref{lemma:thicknessineq}, we have the following theorem.

\begin{theorem}[Stipsicz and Szabó, in preparation]\label{thm:baddomainbound}
Suppose that $B(D)$ denotes the number of bad domains in a diagram $D$, representing the knot $K$. Then, $th(K) \leq \frac{1}{2}B(D)$.
\end{theorem}

Our final sufficient condition relies upon the notions of the thickness $th(K)$ and Seifert genus $g(K)$ of a knot $K$, which are only lightly treated here. The reader is referred to \cite{Hanselman} for further details. 

\begin{theorem}[Hanselman, \cite{Hanselman}]\label{thm:Hanselman} 
If a nontrivial knot $K$ has $th(K)\leq 5$ and $g(K) \neq 2,$ then the PCSC holds for $K.$
\end{theorem}

\section{Verification of the conjecture}
Let us first introduce the Kinoshita-Terasaka knots, and record some basic facts about them. These knots are parametrized by integers $r,n,$ with $KT_{r,n}$ denoting the knot of the form shown in Figure \ref{fig:KT}. It can be shown that $K_{r,n}$ is isotopic to the unknot if and only if $r \in \{0, 1, -1, -2 \}$ or $n=0$. Also, by turning the knot inside out, one can observe a symmetry which identifies 
\begin{equation} \label{KTsym}
    KT_{r,n}= KT_{-r-1,n}.
\end{equation}
Finally, we note that the reflection of $KT_{r,n}$ is the knot $KT_{r,-n}$. For brevity's sake, we will not say more about these knots; however no other facts than these are used in our proof. The reader is referred to \cite{Gabai}, \cite{KinoshitaTerasaka}, and \cite{OzsvathAndSzaboTwo} for further information. 

\begin{figure}
    \centering
    \includegraphics[scale=0.4]{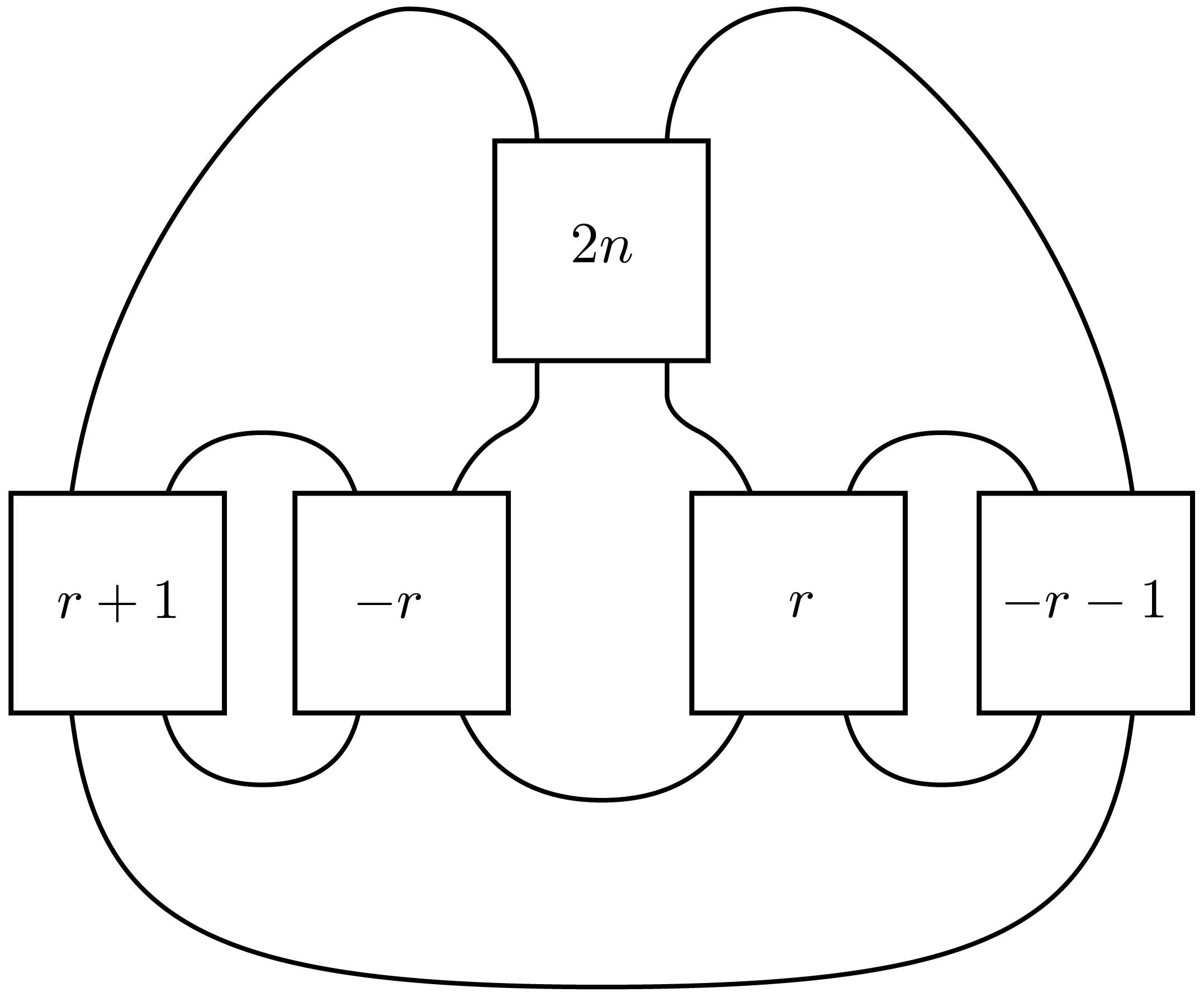}
    \caption{The Kinoshita-Terasaka knot family.}
    \label{fig:KT}
\end{figure}

Now, we turn to the Conway knots, which are also parametrized by integers $r,n,$ with $C_{r,n}$ denoting the knot of the form shown in Figure \ref{fig:Conway}. Additionally, Kinoshita-Terasaka and Conway knots satisfy the many of the same relations; in particular we have that 
\begin{equation} \label{Conwaysym}
    C_{r,n}= C_{-r-1,n}
\end{equation}
and that the reflection of $C_{r,n}$ is the knot $C_{r,-n}$. Also, for $r \in \{ 0, 1, -1, -2\}$ or $n=0$, $C_{r,n}$ is isotopic to the unknot, and hence is excluded in our discussion of the PCSC.
\begin{figure}
    \centering
    \includegraphics[scale=0.4]{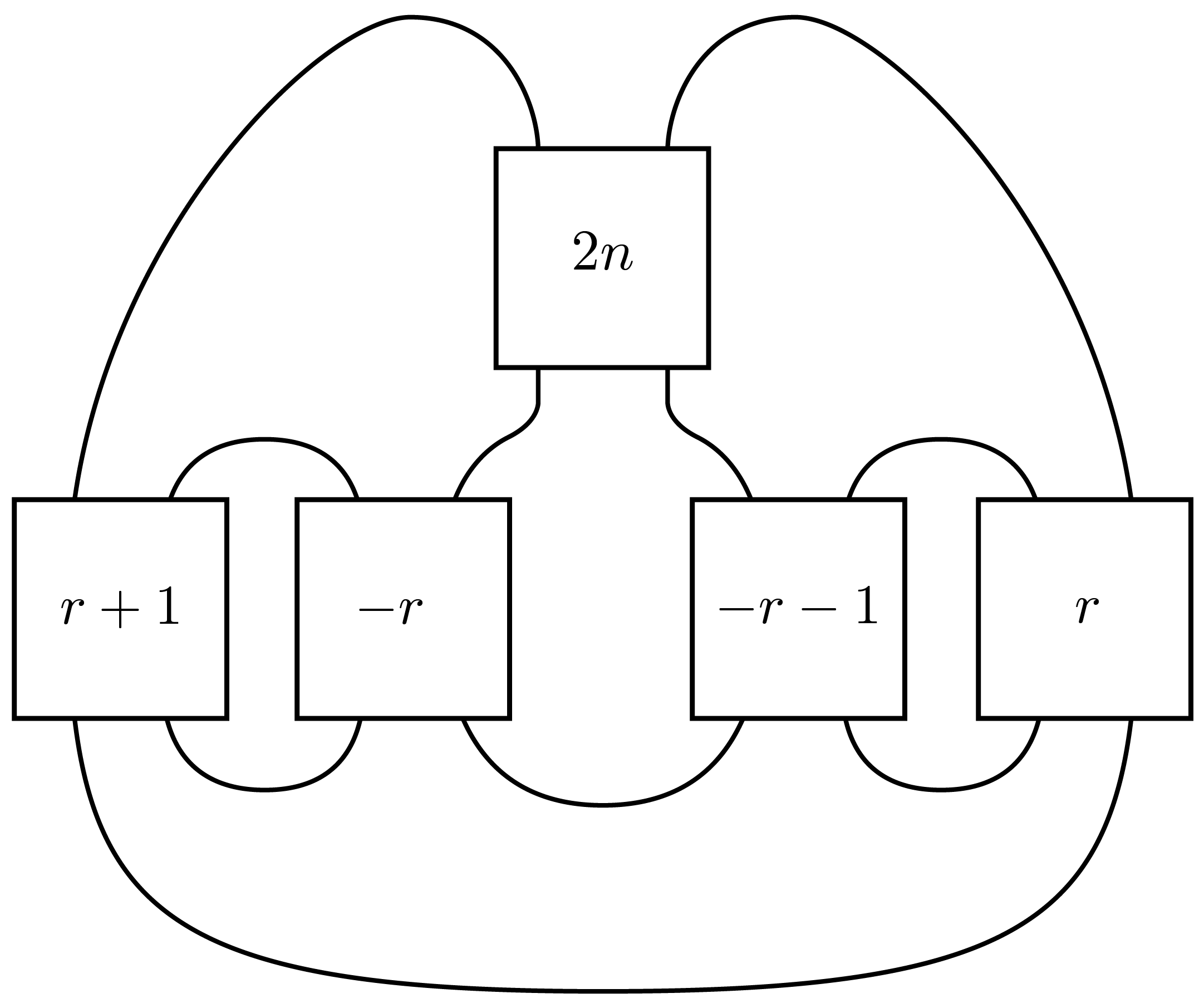}
    \caption{The Conway knot family.}
    \label{fig:Conway}
\end{figure}
With the Kinoshita-Terasaka and Conway knots described, we can begin the proof of Theorem \ref{thm:PCSCKTConway}.

\subsection{Thickness of the Kinoshita-Terasaka and Conway knots}\label{thickness}
With a view towards applying the theorem of Hanselman (Theorem \ref{thm:Hanselman}), we can begin by bounding the thickness of the Kinoshita-Terasaka and Conway knot families by above. 
\begin{proposition}\label{KTConwayThickness}
The thickness of the Kinoshita-Terasaka and Conway knots is at most 2.
\end{proposition}
\begin{proof}
\begin{figure}
    \centering
    \includegraphics[scale=0.4]{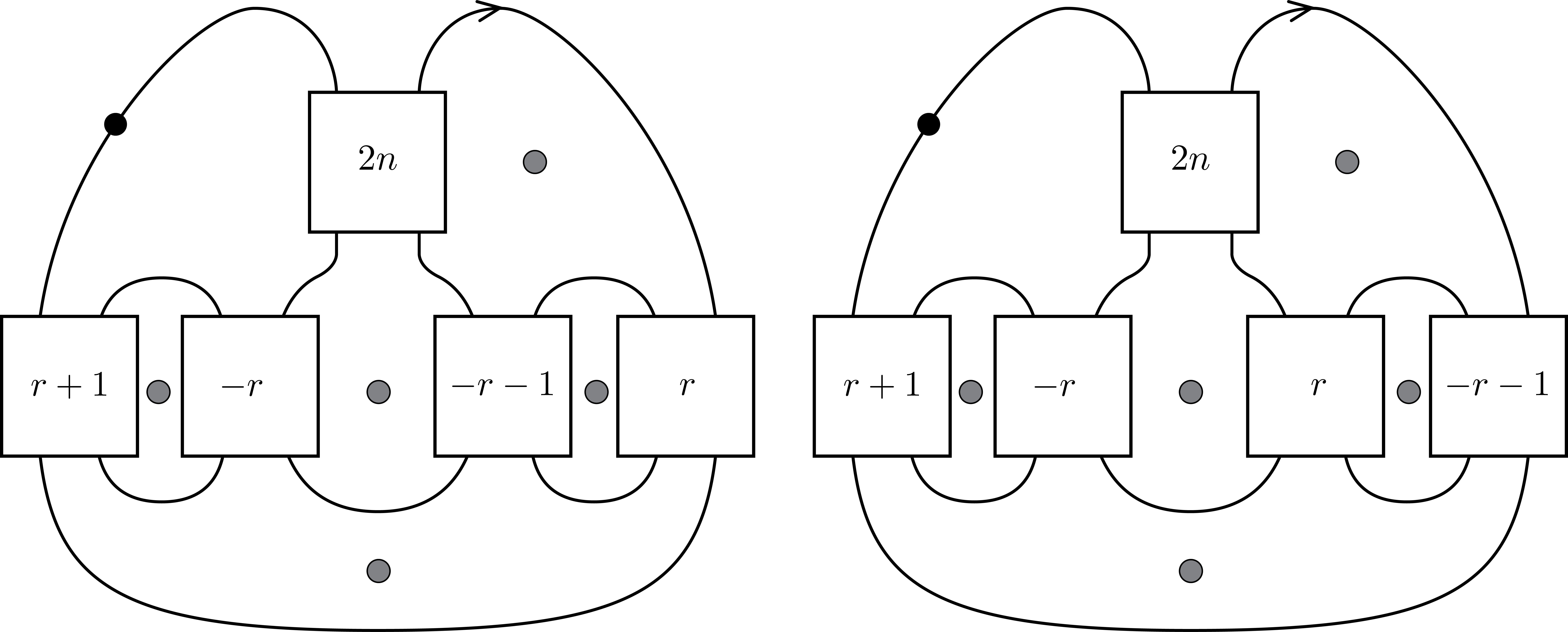}
    \caption{Marked diagrams of the Kinoshita-Terasaka and Conway knot families with domains marked in gray.}
    \label{fig:ktconwaymarked}
\end{figure}
Consider the marked diagrams for the Kinoshita-Terasaka and Conway knot families as illustrated in Figure \ref{fig:ktconwaymarked}. We can examine that for $n \neq 0$, $r>1$ (and hence $r<-2$ by symmetry, the values for which the knots are nontrivial), we have five marked domains, as well as those domains enclosed by the adjacent half-twists. Regardless of the orientation of the adjacent half-twists within the boxes, we have that any domain with the half-twist as a boundary is a good domain. Therefore, we have that $B(D) \leq 5$ from the five possible remaining domains. By Theorem \ref{thm:baddomainbound}, we have that $th(K) \leq \frac{1}{2} \cdot 5$. In particular, since the thickness of a knot must be integer-valued, we have that $th(K) \leq 2$.
\end{proof}

\subsection{Seifert genera of the Kinoshita-Terasaka and Conway knots}
The final element needed to complete our proof of Theorem \ref{thm:PCSCKTConway} is the notion of Seifert genus, which we will use to apply Hanselman's theorem to our knots. The Seifert genus $g(K)$ is defined as the minimal genus of a Seifert surface for the knot $K$. Using this notation, we have the following theorem of Gabai:
\begin{theorem}[Gabai, \cite{Gabai}]
For $r > 0,$ the Kinoshita-Terasaka knot $KT_{r,n}$ has Seifert genus $g(KT_{r,n}) =r$.
\end{theorem}
In \cite{OzsvathAndSzaboTwo}, Ozsváth and Szabó used knot Floer homology to provide an alternate proof of this result, as well as describe the Seifert genera of Conway knots in the following way:
\begin{theorem}[Ozsváth and Szabó, \cite{OzsvathAndSzaboTwo}]
For $r > 0,$ the Conway knot $KT_{r,n}$ has Seifert genus $g(C_{r,n}) = 2r-1$.
\end{theorem}
\begin{lemma} No non-trivial Conway knot admits a purely cosmetic surgery. Furthermore, for $r\neq 2,$ no nontrivial Kinoshita-Terasaka knot $KT_{r,n}$ admits a purely cosmetic surgery.
\end{lemma}
\begin{proof}
As a consequence of the theorems of Gabai and Ozsváth and Szabó, we note by applying the symmetries from Equations \ref{KTsym} and \ref{Conwaysym} that all nontrivial Kinoshita-Terasaka knots with $r \neq 2$ (and hence $r \neq -3$) and all nontrivial Conway knots have $g(K) > 2$. Combining Theorem \ref{thm:Hanselman} and Proposition \ref{KTConwayThickness}, we thus have that no purely cosmetic surgeries are admitted by $C_{r,n}$ for all $r,n\in \mathbb{Z},$ nor by $KT_{r,n}$ for $r\neq 2,$ $n\in \mathbb{Z}.$
\end{proof}
We handle the remaining case $r=2$ in the following section.

\subsection{Verification for \texorpdfstring{$KT_{2,n}$}{Lg}}

Notice that $KT_{2,0}$ is isotopic to the unknot $u$, so it remains to prove that the PCSC holds for $KT_{2,n}$ for any $n\in\mathbb{Z}^{*}$. We will show this by checking the sufficient condition given in Theorem \ref{IchiharaWuThm}.

\begin{lemma} \label{lemma:nonzerothirdderiv}
$V_{K T_{2, n}}^{\prime \prime \prime}(1) \neq 0 \quad$ for any $n \in \mathbb{N}^{*}$.
\end{lemma}
\begin{proof}
\begin{figure}
	\centering
	\includegraphics[scale=0.4]{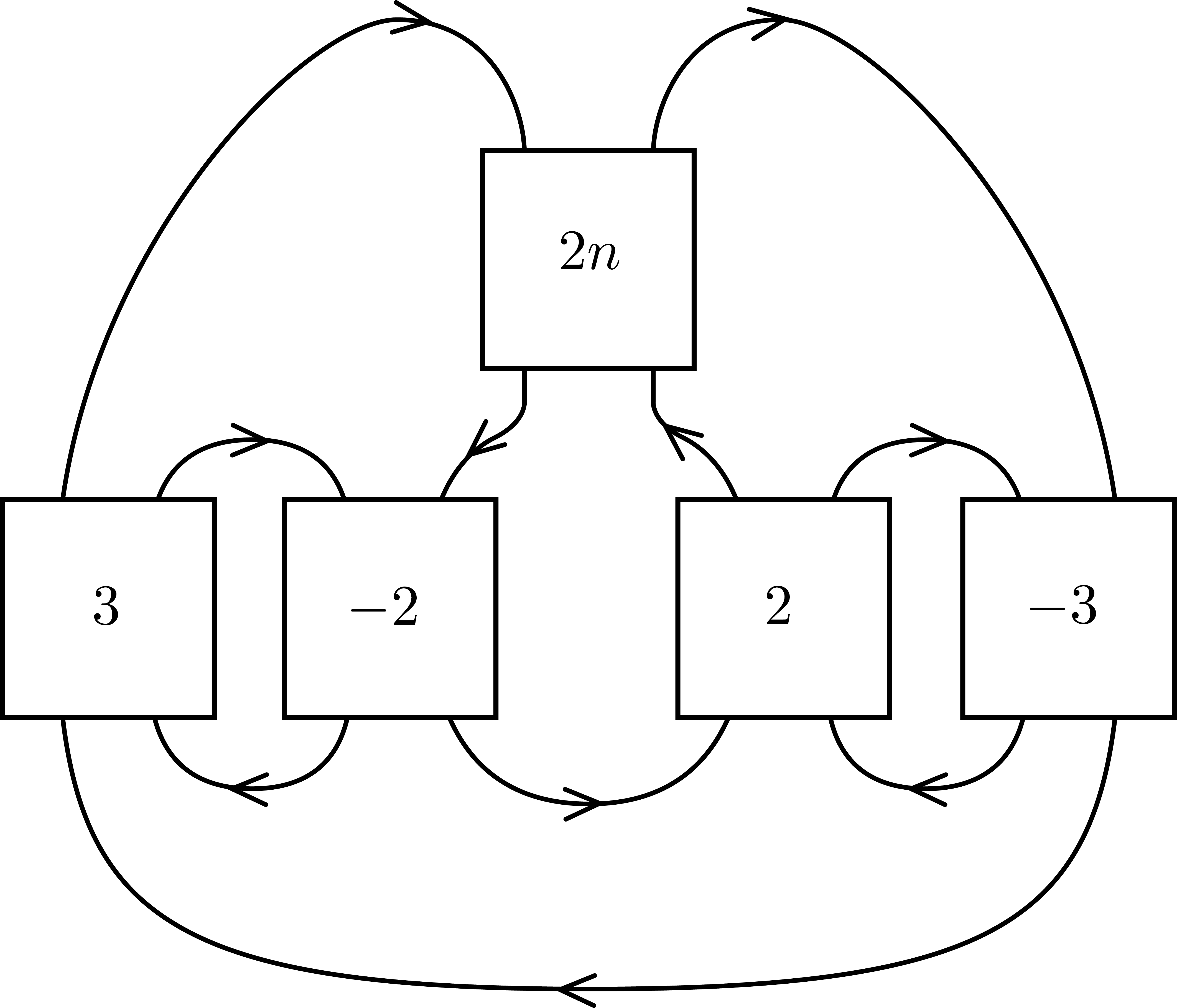}
	\caption{$KT_{2,n}$.}
	\label{fig:KT2n}
\end{figure}
Let $n$ be a non-negative integer, and orient $KT_{2,n}$ as shown in Figure \ref{fig:KT2n}. Consider the crossing on the top of $KT_{2,n}$ and consider the skein triple with $\Vec{L}_{+}\ :=\ KT_{2,n}$. It follows that  $\Vec{L}_{-}=KT_{2,n-1}$ and $\Vec{L}_0=P(3,-2,2,-3)$, where $P(3,-2,2,-3)$ is a pretzel link oriented as in Figure \ref{fig:pretzel}.

\begin{figure}
	\centering
	\includegraphics[width=0.8\textwidth]{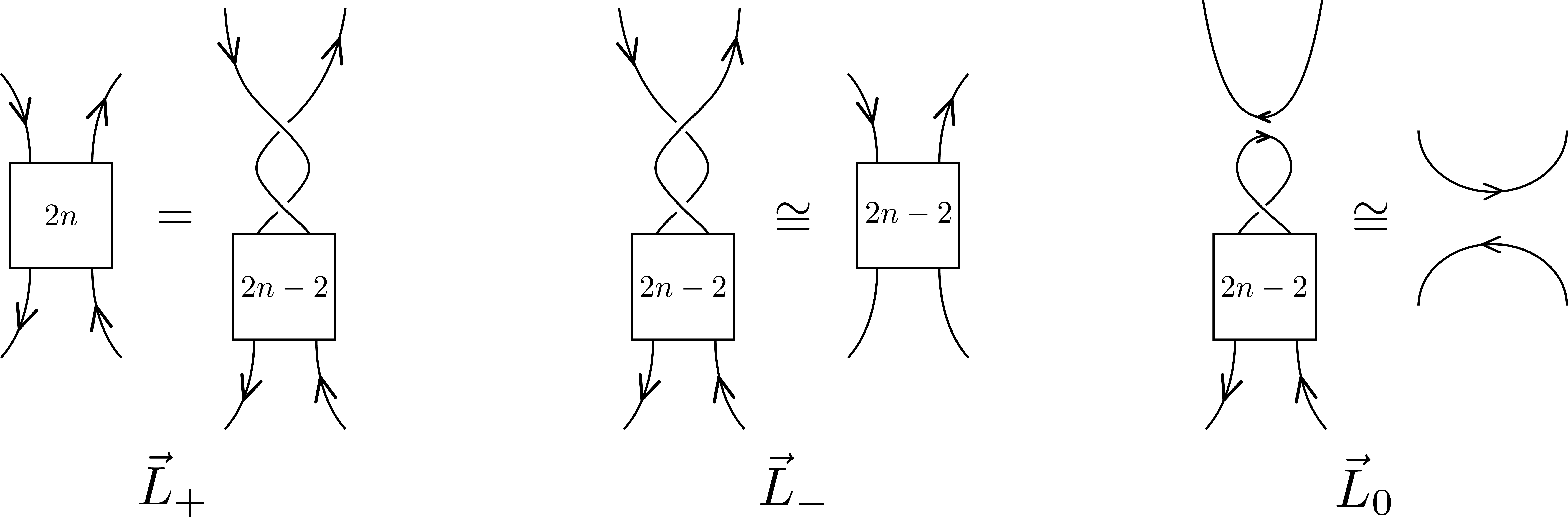}
	\caption{A skein triple for $KT_{2,n}$.}
\end{figure}
\begin{figure}
	\centering
	\includegraphics[scale=0.4]{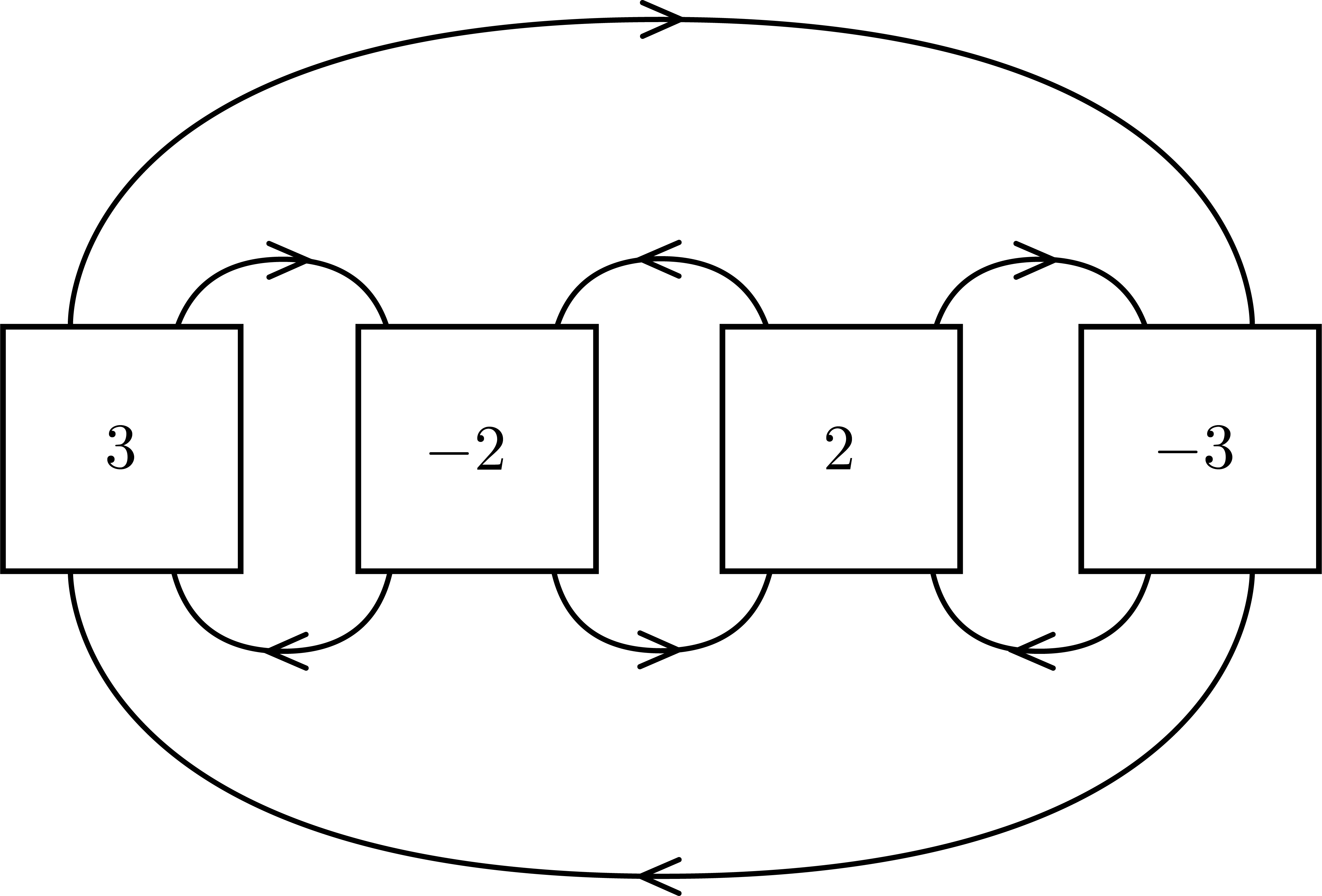}
	\caption{$P(3,-2,2,-3)$.}
	\label{fig:pretzel}
\end{figure}

Denote $V_{n}\ :=\ V_{KT_{2, n}}$ and $V\ :=\ V_{P(3,-2,2,-3)}$. By the skein relation, we have that 
$$
q^{-2} V_{n}-q^{2} V_{n-1}=\left(q^{-1}-q\right) V,
$$
\text{and equivalently that}
\begin{equation} \label{LDE}
V_{n}-q^{4} V_{n-1}=\left(q-q^{3}\right) V.
\end{equation}

Solving the linear difference equation \ref{LDE} with boundary condition
$$
V_{0}=V_{K T_{2,0}}=V_{u}=1,
$$

we obtain that 
\begin{equation} \label{Sol1}
V_{n}=\left(1-\dfrac{V}{q+q^{-1}}\right) q^{4 n}+\dfrac{V}{q+q^{-1}}=q^{4 n}+\dfrac{1-q^{4 n}}{q+q^{-1}} V.
\end{equation}

Computing the Jones polynomial of the pretzel link using the skein relation, we have that
$$
V=-q^{9}+q^{7}-q^{5}+q^{3}+q^1+q^{-1}+q^{-3}-q^{-5}+q^{-7}-q^{-9}.
$$

Also notice that 
$$
\begin{aligned}
1-q^{4 n}&=\left(q+q^{-1}\right) \sum_{k=1}^{2 n}(-1)^{k} q^{4 n-2 k+1}\\ &=\left(q+q^{-1}\right) \sum_{k=1}^{n}\left(q^{4n-4 k+1}-q^{4 n-4 k+3}\right)\\ &=\left(q+q^{-1}\right) \sum_{k=1}^{n}\left(q^{4k-3}-q^{4 k-1}\right).
\end{aligned}
$$

If we let $P_{n}\ :=\ \sum_{k=1}^{n}\left(q^{4 k-3}-q^{4k-1}\right)$, then Equation $\ref{Sol1}$ can be written as 
\begin{equation} \label{Sol2}
V_{n}=q^{4 n}+V\cdot P_{n}.\\
\end{equation}

Differentiating Equation $\ref{Sol2}$ using Leibniz's rule, we have that 
$$
V_{n}^{\prime \prime \prime}=\left(q^{4 n}\right)^{\prime \prime \prime}+V^{\prime \prime \prime} P_{n}+3 V^{\prime \prime} P_{n}^{\prime}+3 V^{\prime} P_{n}^{\prime \prime}+V P_{n}^{\prime \prime \prime}.
$$

Also, direct calculation of the derivatives of $V$ and $P_n$ at $q=1$ gives that 
$$
V(1)=2,\ V^{\prime}(1)=0,\ V^{\prime \prime}(1)=-94
$$ 
and 
$$
P_{n}(1)=0,\quad P_{n}^{\prime}(1)=\sum_{k=1}^{n}(4 k-3-(4 k-1))=-2 n,
$$
$$
\begin{aligned}
P_{n}^{\prime \prime \prime}(1) &=\sum_{k=1}^{n}((4 k-3)(4 k-4)(4 k-5)-(4 k-1)(4 k-2)(4 k-3)) \\
&=-6 \sum_{k=1}^{n}(4 k-3)^{2} \\
&=-64 n^{3}+48 n^{2}+4 n.
\end{aligned}
$$

Therefore, we conclude that
$$
\begin{aligned}
V_{n}^{\prime \prime \prime}(1) &=4 n(4 n-1)(4 n-2)+3 \cdot(-94) \cdot(-2 n)-64 n^{3}+48 n^{2}+4 n \\
&=64 n^{3}-48 n^{2}+8 n+564 n-64 n^{3}+48 n^{2}+4 n \\
&=576 n.
\end{aligned}
$$

This implies that $V_{n}^{\prime \prime \prime}(1) \neq 0,$ that is, $V_{K T_{2, n}}^{\prime \prime \prime}(1) \neq 0$ for any $n \in \mathbb{N}^{*}.$
\end{proof}
\begin{lemma} \label{lemma:mirror}
$V_{KT_{2,-n}}^{\prime \prime \prime}(1)=-V_{KT_{2,n}}^{\prime \prime \prime}(1)$ for any $n \in \mathbb{Z}^{*}$.
\end{lemma}
\begin{proof}
Take $n \in \mathbb{Z}^+$. Since $KT_{2,-n}$ is the mirror image of $KT_{2,n}$, we have that $V_{K T_{2, -n}}(q)=V_{K T_{2, n}}\left(q^{-1}\right)$.

Calculating $V^{\prime \prime \prime}_{KT_{2,-n}}(q)$ using the chain rule, we obtain that 
$$
V^{\prime \prime \prime}_{KT_{2, -n}}(q)=-6 q^{-4} V^{\prime}_{KT_{2, n}}\left(q^{-1}\right)-6 q^{-5} V^{\prime \prime}_{KT_{2,n}}\left(q^{-1}\right)-V^{\prime \prime \prime}_{KT_{2, n}}\left(q^{-1}\right).
$$

Therefore,
\begin{equation} \label{M1}
V^{\prime \prime\prime}_{KT_{2,-n}}(1)=-6V^{\prime}_{KT_{2,n}}(1)-6V^{\prime \prime}_{KT_{2,n}}(1)-V^{\prime \prime\prime}_{KT_{2,n}}(1).
\end{equation}

Using the results from Lemma \ref{lemma:nonzerothirdderiv}, we have that
\begin{equation} \label{M2}
V^{\prime}_{KT_{2,n}}(1)\equiv V^{\prime}_n(1)=4n+V^{\prime}(1)P_n(1)+V(1)P^{\prime}_n(1)=4n+0\cdot 0+2(-2n)=0.
\end{equation}

Also, since $\Delta_{KT_{2,n}}\equiv 1$, we have that $\Delta^{\prime\prime}_{KT_{2,n}}(1)= 0$. Notice that from Lemma \ref{relationship} we have the relation $V_{KT_{2,n}}^{\prime\prime}(1)=-3\Delta_{KT_{2,n}}^{\prime \prime}(1)$, we obtain that 
\begin{equation} \label{M3}
V_{KT_{2,n}}^{\prime \prime}(1)=0.
\end{equation}

Together, Equations \ref{M1}, \ref{M2} and \ref{M3} imply that for any $n \in \mathbb{Z}^+$,
\begin{equation}\label{M4}
V^{\prime\prime\prime}_{KT_{2,-n}}(1)=-V^{\prime\prime\prime}_{KT_{2,n}}(1).
\end{equation}

Due to the symmetry of identity \ref{M4}, we conclude that 
$V^{\prime\prime\prime}_{KT_{2,-n}}(1)=-V^{\prime\prime\prime}_{KT_{2,n}}(1)$ for any $n\in\mathbb{Z}^{*}.$
\end{proof}
\begin{lemma}
The PCSC holds for $KT_{2,n}$ for any $n\in\mathbb{Z}^{*}$.
\end{lemma}

\begin{proof}
From Lemma \ref{lemma:nonzerothirdderiv} we have that $V^{\prime\prime\prime}_{KT_{2,n}}(1)\neq 0$ for any $n\in\mathbb{Z}^+$. Take any $n\in\mathbb{Z}^-$, by Lemma \ref{lemma:mirror} we have that $V^{\prime\prime\prime}_{KT_{2,n}}(1)=-V^{\prime\prime\prime}_{KT_{2,-n}}(1)\neq 0$. This implies that $V^{\prime\prime\prime}_{KT_{2,n}}(1)\neq 0$ for any $n\in\mathbb{Z}^{*}$. Therefore, from Theorem \ref{IchiharaWuThm} we conclude that the PCSC holds for $KT_{2,n}$ for any $n\in\mathbb{Z}^{*}$.
\end{proof}
This computation completes our proof of Theorem \ref{thm:PCSCKTConway}.

\printbibliography

\vspace{10pt}
\noindent \textsc{Department of Mathematics and Statistics, Carleton College, Northfield, MN 55057} 

\noindent \textit{E-mail address:} \href{mailto:boehnkeb@carleton.edu}{\nolinkurl{boehnkeb@carleton.edu}}
\bigskip

\noindent \textsc{Department of Mathematics, Cornell University, Ithaca, NY, 14853} 

\noindent \textit{E-mail address:} \href{mailto:cg527@cornell.edu}{\nolinkurl{cg527@cornell.edu}}
\bigskip

\noindent \textsc{Department of Mathematical Sciences, Loughborough University, Loughborough, LE11 3TU, United Kingdom} 

\noindent \textit{E-mail address:} \href{mailto:H.Liu6-19@student.lboro.ac.uk}{\nolinkurl{H.Liu6-19@student.lboro.ac.uk}}
\bigskip

\noindent \textsc{Department of Mathematics and Statistics, Carleton College, Northfield, MN 55057} 

\noindent \textit{E-mail address:} \href{mailto:xues@carleton.edu}{\nolinkurl{xues@carleton.edu}}

\end{document}